\documentclass[a4paper,12pt]{article}
\usepackage{lscape}
\usepackage{a4,amssymb,amsmath,amsthm,url,lineno,threeparttable}
\usepackage[normalem]{ulem}
\usepackage{bezier,amsfonts,graphicx,url}
\usepackage[utf8]{inputenc}
\usepackage[T1]{fontenc}
\usepackage[english]{babel}
\usepackage{color}

\title{On zero-sum spanning trees and zero-sum connectivity}
\date{}
\begin{document}
\newtheorem{theorem}{Theorem}[section]
\newtheorem*{theoremnonum}{Theorem}
\newtheorem{definition}{Definition}[section]
\newtheorem{conjecture}[theorem]{Conjecture}
\newtheorem{proposition}[theorem]{Proposition}
\newtheorem{corollary}[theorem]{Corollary}
\newtheorem{lemma}[theorem]{Lemma}
\newtheorem{example}[theorem]{Example}
\newtheorem*{problem}{Problem}
\newcommand{\bal}{{\rm bal}}
\newcommand{\sbal}{{\rm sbal}}
\newcommand{\ot}{{\rm ot}}
\newcommand{\ex}{{\rm ex}}
\newcommand{\diam}{{\rm diam}}
\newcommand{\tot}{{\rm tot}}
\newcommand{\half}{{\rm Half}}

\newcommand{\ad}[1]{\textcolor{blue}{#1}}

\DeclareGraphicsExtensions{.pdf,.png,.jpg}

\author{Yair Caro \\ Department of Mathematics\\ University of Haifa-Oranim \\ Israel \and Adriana Hansberg \\  Instituto de Matem\'{a}ticas \\ UNAM Juriquilla \\ Quer\'e{}taro, Mexico\\ \and \ Josef  Lauri\\ Department of Mathematics \\ University of Malta
\\ Malta \and Christina Zarb \\Department of Mathematics \\University of Malta \\Malta  }

\maketitle

\begin{abstract}
	We consider $2$-colourings $f :  E(G)  \rightarrow \{ -1  ,1 \}$ of the edges of a graph $G$ with colours $-1$ and $1$ in $\mathbb{Z}$. A subgraph $H$ of $G$ is said to be a zero-sum subgraph of $G$ under $f$ if $f(H) := \sum_{e\in E(H)} f(e) =0$. We study the following type of questions, in several cases obtaining best possible results: Under which conditions on $|f(G)|$ can we guarantee the existence of a zero-sum spanning tree of $G$? The types of $G$ we consider are complete graphs, $K_3$-free graphs, $d$-trees, and maximal planar graphs. We also answer the question of when any such colouring contains a zero-sum spanning path or a zero-sum spanning tree of diameter at most $3$, showing in passing that the diameter-$3$ condition is best possible. Finally, we give, for $G = K_n$, a sharp bound on $|f(K_n)|$ by which an interesting zero-sum connectivity property is forced, namely that any two vertices are joined by a zero-sum path of length at most $4$.   
	
	One feature of this paper is the proof of an Interpolation Lemma leading to a Master Theorem from which many of the above results follow and which can be of independent interest. 
\end{abstract}

	\section{Introduction}
	Perhaps the most obvious point of reference for the origin of problems on zero-sum trees is the almost trivial observation that, given a graph $G$, either $G$ or its complement $\overline{G}$, is connected, that is, has a spanning tree.  A two-colouring formulation of this problem gives it a Ramsey-theoretic flavour: if the edges of complete graph $K_n$ are coloured with two colours, then $K_n$ must have a monochromatic spanning tree. This quest for a monochromatic spanning tree has been modified in various ways which turn simple results like this into interesting research questions. For example, one can let the colours be $0$ and $1$ and then require that the sum of the colours on the edges of the spanning tree is even, making it a problem in zero-sum Ramsey-theory over $\mathbb{Z}_2$, which is completely solved \cite{caro1994complete, caro1999uniformity, CaroYuster1998, WilsonWong2014}. The problem now takes on a more decidedly Ramsey-theoretic nature because the right question to ask would be: is there an $N$ such that, for all $n>N$, $n$ odd, any $0$-$1$ colouring of $E(K_n)$ contains a spanning tree such that the sum of the colours on the edges of the spanning tree is even. The extension to colourings with elements of $\mathbb{Z}_n$, the cyclic group of integers modulo $n$, and requiring the sum of the colours on the edges of the spanning tree to be equal to  $0 \bmod n$  is then clear. 
	
	In this paper, we shall consider various variations of this problem, always asking that the required subtree is zero-sum (or almost zero-sum when appropriate - to be defined later), that is, the sum of colours on its edges is zero in some domain. In fact, we will study the variant when the edges of the complete graph are coloured with colours $-1$ and $1$, discovering, on the way, some surprising differences between this case and the case when the colours used are $-1, 0$ and $1$. 
	
	Another direction which we shall investigate is the same as above but just changing the nature of the subtree required, for example asking for a spanning path, or a spanning subtree of diameter $3$. We shall also consider, in what we call zero-sum connectivity, namely that any two vertices of $K_n$ are joined by a zero-sum path of length at most $4$. 
	
	Another variation we consider is when the host graph, whose edges we are colouring, is not the complete graph. In this context, we consider complete bipartite graphs, maximal planar graphs, and maximal $d$-degenerate graphs.
	
	Many of these results can stand alone as a separate instance of the problem of finding zero-sum spanning subtrees. However, one of the main aims of this paper is also to illustrate some main techniques which unite them. We therefore consider Section 3 
to be crucial in this paper, where we
prove an Interpolation Lemma, and a what we call the Master Theorem, from which many of the results in the other sections follow and which might also have independent interest.
	
	In the next section, we shall give the main definitions we will require, along with some background in the form of existing results, which will also serve as a benchmark and guideline for our main zero-sum results in this paper.
	
	\section{Main definitions, notation and some background}
	
Let $G$ be a graph and let $f:E(G)  \to A$  be a mapping with range values in an abelian group $A$. For a subgraph $H \le G$, we set $f(H) = \sum_{e \in E(H)} f(e)$ as the \emph{weight} of $H$ under $f$, where the sum is taken in $A$. For the case that  $A = \mathbb{Z}_k$, the cyclic group on $k$ elements, and $H$ is a graph such that $k$ divides $e(H)$, we say that $K_n$ has a \emph{zero-sum $H$} under $f$ if $f(H) := \sum_{e  \in E(H)} f(e) \equiv   0 \pmod k$. The least $n$ such that $K_n$ contains a zero-sum $H$ under $f$ for every $f:E(K_n)  \rightarrow \mathbb{Z}_k$ is denoted by $R(H,\mathbb{Z}_k)$ and is called the \emph{$\mathbb{Z}_k$-zero-sum Ramsey number of $H$} (for a survey see \cite{caro1996zero}).
	
The particular case that $H$ is a spanning tree of $K_n$ was one of the first problems raised in the emergence of zero-sum Ramsey theory around 1990. The question of the existence of a zero-sum$\pmod n$ spanning tree for every $f:E(K_{n+1}) \rightarrow \mathbb{Z}_n$,  was solved affirmatively in the case when $n$ is prime in \cite{bialostocki1990zero}, and for every $n$ in \cite{furkleit}, and in further generality in  \cite{schrijver1991simpler}. 
	
	Also, in classical Ramsey theory, much research has been done trying  to gain more knowledge about the forced structure of monochromatic spanning trees in an edge-coloured complete graph and, in particular, it is known that there is always  a monochromatic spanning tree of height two (hence diameter four),  as well as  spanning trees called brooms \cite{bialostocki2001either,gyarfas2011large}.	However, efforts in this vein to get some further knowledge on the structure
	of forced zero-sum$\pmod n$ spanning trees seems to not have been developed further. 

The appearance of new versions of zero-sum problems, where the range set is not ${\mathbb Z}_k$, but elements in $\mathbb{Z}$,  mostly $\{ -1,0,1\}$ or $\{ -1,1\}$,  started in \cite{caro2016zero} (see also \cite{caro2018unavoidable,caro2019zero2,caro2019zero,caro2019small}), although an early paper about possible weights of spanning trees of the $n$-dimensional cube $Q_n$ under  a $\{ -1 ,1\}$-colouring of its edges appeared in \cite{HARARY199385}. One of the first questions considered was:  under what conditions  does $f:E(K_n) \rightarrow  \{ -1,0,1\}$  force a zero-sum (over $\mathbb{Z}$)  spanning tree?  The exact solution is:  whenever  $f : E(K_n) \rightarrow \{ -1,0,1\}$ is such that $n \equiv 1  \pmod 2$  and   the absolute weight $|f(K_n)|  \leq  n - 2$,  then there exists a zero-sum (over $\mathbb{Z}$) spanning tree, and this bound is sharp \cite{caro2016zero}.\\
	
\noindent 	
In this paper, we shall concentrate on colouring the edges of a graph $G$, mostly the complete graph $K_n$, with colours $-1$ and $1$.  To consider this problem, we need to introduce some further notation which will be used in the sequel.  For $f :  E(G)  \rightarrow \{ -1  ,1 \}$  we write $E(-1)  = \{  e  \in E(G) :  f(e)  = -1  \}$, $E(1) =  \{  e \in   E(G)  :  f( e)  = 1  \}$, $e(-1) =  |E(-1)|$ and $e(1) = |E(1)|$. 
If $f(H) = 0$, we say that $H \le G$ is \emph{zero-sum} under $f$, while if  $|f(H)| = 1$, we say that $H$ is \emph{almost zero-sum}. Of course, the first can only happen if $e(H) \equiv 0 \pmod 2$ and the latter only if $e(H) \equiv 1 \pmod 2$.

The structure of this paper is briefly as follows.	In the next section, we shall prove important results which will underpin most of the the particular zero-sum results which will be presented further on. Analogous to the above situation  with colours$-1$, $0$ and $1$,  we shall consider in Section 4 the existence of zero-sum spanning trees for edge-colourings $f :  E(K_n) \rightarrow \{ -1 ,1\}$. As we shall see, somewhat counterintuitively,  the absence of $0$ in the range of $f$ forces a much less tight bound on  $|f( K_n)|$ than the case where $0$ is allowed, namely $|f(K_n)|  \leq \binom{n}{2} - 2 \binom{(n-1)/2}{2}  = \frac{( n-1)(n-3)}{4}$.  Observe that a condition of type $|f(K_n)|  \leq h(n)$ is equivalent to asking $\min\{e(-1), e(1)\} \ge \frac{1}{2} \left({n \choose 2} - h(n) \right)$. Hence, the appearance of conditions on $\min \{ e(-1)  , e( 1) \}$ is typical of all zero-sum problems over $\mathbb{Z}$  and the effect of  the difference  between  the range $\{ -1,0 ,1\}$  and $\{ -1,1\}$  has been already shown to be somewhat dramatic \cite{caro2019zero2}.  We also give in this section a sharp result on zero-sum spanning paths and spanning trees with diameter at most $3$ in a complete graph whose edges are coloured with $-1$ and $1$. This result is also inspired by one of the starting points of this paper, namely the folklore  variants of the result: if $G$ is a graph of diameter $\diam(G) \geq 4$,  then $\diam(\overline{G})  \leq 2$ \cite{harary1985diameter}.  
	
In Section 5, we shall study the analogous case of zero-sum spanning trees  for $\{-1,1\}$-colourings of biparite graphs, $d$-trees and maximal planar graphs, giving in all three cases best possible bounds.
		
	In Section 6, we shall consider zero-sum connectivity, where we will require that any two vertices of $K_n$ are joined by a zero-sum path of length at most 4.
	
	Finally, in the concluding section, we shall present some ideas for further investigation.



\section{The master theorem for zero-sum and almost-zero sum spanning subgraphs}
\subsection{Definitions and examples}

We give some definitions and results which will be used in the sequel.

\medskip

\noindent \textbf{Edge replacement}\\
 Given a subgraph $H$ of a graph $G$, we say that a subgraph $H'$ of $G$ is obtained  by an \emph{edge-replacement} from $H$ if there are edges $e \in E(H)$ and $e' \in E(H')\backslash E(H)$ such that $E(H')= \left( E(H) \backslash \{e\}\right) \cup \{e'\}$.

\medskip

\noindent \textbf{Closed family}\\
 A family $\mathcal{F}$ of subgraphs of a graph $G$ is called a \emph{closed family in $G$}  if,  for  any two subgraphs $H$ and $H'$ of $G$ which are isomorphic to members of $\mathcal{F}$, there is a chain $H = H_1, H_2, \ldots, H_q = H'$ of subgraphs of $G$, each one isomorphic to some member of $\mathcal{F}$, such that, for $1 \le i \le q-1$, $H_{i+1}$ is obtained from $H_i$ by an edge-replacement.\\

When the family $\mathcal{F}$ is closed in the graph  $K_n$ we shall just say  that $\mathcal{F}$ is a \emph{closed family}, omitting the host graph $K_n$.

A classic example of a closed family is  the family of spanning trees of a connected graph $G$, which form 
the basis of the so called Cycle Matroid of $G$. The next lemma is taken from \cite{welsh2010matroid}.

\begin{lemma} \label{matroids}
Let $G$ be a connected graph. Then the following statements are valid.
\begin{enumerate}
\item The family of all forests contained in $G$ forms the Cycle Matroid denoted by $M(G)$.
\item The basis of this matroid is the set of all spanning trees of $G$.
\item For any two spanning trees $T$ and $T'$ of $G$, there is a chain of spanning trees $T = T_0 , T_1, \ldots,T_q = T'$,  such that , for $1 \le i \le q-1$, $T_{i+1}$ is obtained from $T_i$ by an edge replacement.
\end{enumerate}
\end{lemma}

A graph $G$ on $n$ vertices  is called a {\em local amoeba} if, for any two copies $H$ and $H'$ of $G$ in $K_n$, there is a chain $H=G_0,G_1,\ldots,G_q=H'$ such that, for every $1 \le i \le q-1$, $G_i \cong G$ and $G_{i+1}$ is obtained from $G_i$ by an edge-replacement.   A graph $G$ is called a {\em global amoeba} if there exists an integer $n_0=n_0(G) \geq |V(G)|$, such that for all $n \geq n_0$ and any two copies $H$ and $H'$ of $G$ in $K_n$, there is a chain $H=G_0,G_1,\ldots,G_q=H'$ such that, for every $1 \le i \le q-1$, $G_i\cong G$ and $G_{i+1}$ is obtained from $G_i$ by an edge-replacement.  The notion of amoebas was introduced in \cite{caro2018unavoidable}, further developed in \cite{caro2020amoebas} and also used in \cite{caro2019small}. 

\smallskip




\smallskip

Now we can give examples of closed families:
\begin{enumerate}
\item A local amoeba on $n$ vertices is a closed family with a single element. Examples of such graphs are, to mention some, the path $P_n$, and $K_n - e$, the complete graph minus an edge, for $n \geq 4$.
\item For every $N \geq n$, a global amoeba on $n$ vertices forms a closed family in $K_N$ with a single element. Examples of such graphs are: $P_n$,  $nK_2$,  and the graph consisting of a cycle $C_k$ with a pending edge, to mention some.
\item The family of all connected graph on $n$ vertices and a fixed number of edges is a closed family in $K_n$.
\item The family of all graphs having a Hamiltonian path on $n$ vertices and a fixed number of edges is a closed family in $K_n$.
\end{enumerate}

\smallskip

\noindent \textbf{Covering family}\\
 A family $\mathcal{D}$ of graphs is called a \emph{covering family} of a closed family $\mathcal{F}$ if, for every $H  \in \mathcal{D}$ and every edge $e \in E(H)$, $H- e \in \mathcal{F}$.\\

For example, $\mathcal{D} = \{C_n\}$ is a covering family of the closed family $\mathcal{F} = \{P_n\}$. Also,  the family of all Hamiltonian graphs on $n$ vertices and $m + 1$ edges is a covering family of the closed family of all graphs on $n$ vertices and $m$ edges having a Hamiltonian path.\\

\smallskip

\noindent 
{\bf The family $\boldsymbol{\half(\mathcal{F})}$}\\
For a family $\mathcal{F}$ of graphs on $m$ edges, we define \[\half(\mathcal{F})  =  \left\{   H :  H \le F \mbox{ for some } F \in \mathcal{F},   e(H) = \left\lfloor\frac{m}{2} \right\rfloor, H \mbox{ has no isolates} \right\}.\]

\mbox{}

\noindent 
\textbf{Further basic notation}\\
Let $\mathcal{F}$ be a family of graphs on $m$ edges. A graph $G$ is said to have an \emph{$\mathcal{F}$-decomposition} if the edges of $G$ can be covered by an edge-disjoint union of copies of members of $\mathcal{F}$. The Tur\'an number of $\mathcal{F}$, denoted by $\ex( n,\mathcal{F} )$, is defined as the maximum integer $q$ such that there exists a graph $H$ with $|V(H)|=n$ and $|E(H)|=q$ and no member of $\mathcal{F}$ as a subgraph of $H$. If $\mathcal{F}$ consists of only one graph $F$, we write $\ex( n,F)$ instead of $\ex( n,\mathcal{F} )$.

\subsection{The interpolation lemma}

\begin{lemma}[Interpolation lemma for a closed family on a graph $G$] \label{interpolation}
Let $\mathcal{F}$ be a closed family on a graph $G$ such that its members have each $m$ edges, and let $f: E (G) \rightarrow  \{-1, 1\}$.  Suppose there are two members $H, H' \in \mathcal{F}$,  which are subgraphs of $G$, and assume that $f(H) \leq 0$ and $f(H') \geq 0$. Then there is a subgraph $Z \le G$, $Z \in \mathcal{F}$, which is zero-sum or almost zero-sum under $f$.
\end{lemma}

\begin{proof}
Since $\mathcal{F}$ is closed in $G$, there is a chain of graphs $H  = H_0 , H_1,\ldots,H_q = H'$, where $H_i \le G$  and $H_i \in \mathcal{F}$ for all $1 \le i \le q$, such that, for  $1 \le i \le q-1$, $H_{i+1}$ is obtained from $H_i$ by an edge replacement. Since we remove one edge and insert a new one, we have clearly $|f(H_i) - f(H_{i+1})|  \in \{0, 2\}$, for each $1 \le i \le q-1$. Observe also that all values $f(H_i)$ are of the same parity as $m$. Hence, in the case $m \equiv 0 \pmod 2$,  on the way along the chain from $H$ to $H'$, there must be a $j \in \{1,2 \ldots, q\}$ such that $f(H_j) = 0$.  Similarly, in the case that $m \equiv 1 \pmod 2$  there must be a $j \in \{1,2 \ldots, q-1\}$ such that $f(H_j) = -1$  and $f(H_{j+1}) = 1$. Hence, we have proved in both cases that there is a graph  $Z \le G$, $Z \in \mathcal{F}$ with $|f(Z)| \le 1$, and we are done.
\end{proof}


 We are now ready to prove our main theorem of this  section.

\subsection{The master theorem}

The following theorem deals with three situations where some information is known on a closed family $\mathcal{F}$ in a graph $K_n$ of order $n$: 
\begin{enumerate}
\item{The case when we know the Tur\'an number $\ex(n, \half(\mathcal{F}))$ (or an upper bound on it if $\ex(n, \half(\mathcal{F}))$ is not known).}
\item{The case when $K_n$ has an $\mathcal{F}$-decomposition.}
\item{The case when $K_n$ has a $\mathcal{D}$-decomposition where $\mathcal{D}$ is a covering family of $\mathcal{F}$.}
\end{enumerate}

\begin{theorem}[Master theorem for zero-sum and almost-zero sum spanning subgraphs]\label{MasterThm}
Let $f:E(K_n) \rightarrow \{ -1 ,1\}$ be a colouring of the edges of $K_n$. Let $\mathcal{F}$ be a closed family of graphs on $m$ edges.  Let $c \in \{0,1\}$ be such that $m \equiv c \pmod 2$. Then each one of the following three conditions imply  the existence of a zero-sum or an almost zero-sum spanning graph of $K_n$ which is a member of $\mathcal{F}$.
\begin{enumerate}
\item If $\min\{ e(-1), e(1) \} > \ex(n, \half(\mathcal{F}) )$.
\item If $K_n$ has an $\mathcal{F}$-decomposition, and
\[|f(K_n)| < \frac{2+c}{m} {n \choose 2}.\] 
\item If $\mathcal{D}$ is a covering family of $\mathcal{F}$ such that $K_n$ has a $\mathcal{D}$-decomposition, and
\[|f(K_n)| < \frac{3+c}{m+1}{n \choose 2}.\] 
\end{enumerate}
Moreover, the condition in 1 is best possible.
\end{theorem}

\begin{proof}
\mbox{}\\
1. Since $\min \{ e( -1) ,e(1) \} > \ex(n, \half(\mathcal{F} ) )$, we infer that there is a subgraph $H^{-}$ on $\lfloor\frac{ m}{2} \rfloor$ edges, all of them coloured $-1$, and such that $H^{-}$ is a subgraph of some graph $F^{-}$ isomorphic to some member in $\mathcal{F}$. Analogously, there is a subgraph $H^{+}$ on $\lfloor\frac{ m}{2} \rfloor$ edges, all of them coloured $1$, and such that $H^{+}$ is a subgraph of some graph $F^{+}$ isomorphic to some member in $\mathcal{F}$ (it is possible that $F^-  = F^+$).  Then we have
\[f(F^-) \le \left\lceil\frac{ m}{2} \right\rceil - \left\lfloor\frac{ m}{2} \right\rfloor, \mbox{ and } f(F^+) \ge \left\lfloor\frac{ m}{2} \right\rfloor - \left\lceil\frac{ m}{2} \right\rceil. \] 
When $m \equiv 0 \pmod 2$, then we have $f(F^-) \le 0$ and $f(F^+) \ge 0$, while in the case that $m \equiv 1 \pmod 2$, it follows that $f(F^-) \le 1$ and $f(F^+) \ge -1$. If we have in the latter case that $f(F^-) = 1$ or $f(F^+) = -1$, we are done. So we may assume in both cases that $f(F^-) \le 0$ and $f(F^+) \ge 0$ and thus, by the Interpolation Lemma (Lemma \ref{interpolation}), we conclude that there is a zero-sum or an almost zero-sum spanning graph of $K_n$ which is a member of $\mathcal{F}$.

\smallskip

We will show here also that the bound  $\min\{  e( -1)  , e( 1) \} >  \ex(n,\half(\mathcal{F}) )$ is best possible. To this purpose, we take a colouring  $f:E(K_n) \rightarrow \{ -1 ,1\}$ with $e(-1)  = \ex(n,\half(\mathcal{F}))$ such that $K_n$ does not contain any subgraph $H$ isomorphic to any member of $\half(\mathcal{F})$ with all its edges coloured $-1$. Then $K_n$ can neither contain any subgraph isomorphic to any member of $\mathcal{F}$ with $\left\lfloor\frac{m}{2} \right\rfloor$ edges coloured $-1$, implying that there is no zero-sum or almost zero-sum copy of any member of $\mathcal{F}$.\\

\noindent
2. Since  $K_n$ has  an $\mathcal{F}$-decomposition, it follows that $K_n$  is the union of  $t = \frac{1}{m} {n \choose 2}$ edge-disjoint spanning subgraphs $H_1, H_2, \ldots, H_t$ which are members of $\mathcal{F}$. Recall that $c \in \{0,1\}$ is such that $m \equiv c \pmod 2$. If $f(H_i) \ge 2 + c$ for all $1 \le i \le t$,  or $f(H_i) \le -2 - c$ for all $1 \le i \le t$, then
\[|f(K_n)| = \left|\sum_{i=1}^t f(H_i) \right| \ge t (2+c) = \frac{2+c}{m} {n \choose 2},\]
contradicting the hypothesis. Hence, there are indexes $i_1, i_2$ such that $f(H_{i_1}) \le 1+ c$ and $f(H_{i_2}) \ge -1- c$. If $c = 0$,  since $f(H_i) \equiv 0 \pmod 2$ for all $i$, it follows that  $f(H_{i_1}) \le 0$ and $f(H_{i_2}) \ge 0$. On the other hand, if $c = 1$, we have $f(H_{i_1}) \le 1$ and $f(H_{i_2}) \ge -1$. If, in this case, we have that $f(H_{i_1}) = 1$ or $f(H_{i_2}) = -1$, then we are done. Thus, we can assume that in both cases  $f(H_{i_1}) \le 0$ and $f(H_{i_2}) \ge 0$ holds. Hence, by the Interpolation Lemma (Lemma \ref{interpolation}), it follows that there is a spanning subgraph $H$ which is isomorphic to some member of $\mathcal{F}$ and which is zero-sum or almost zero-sum under $f$.
\medskip

\noindent
3. Since $K_n$ has a $\mathcal{D}$-decomposition, it that follows $E(K_n)$  is the union of  $t = \frac{1}{(m+1)}{n \choose 2}$  edge-disjoint spanning subgraphs $H_1, H_2, \ldots, H_t$ which are members of $\mathcal{D}$. Recall that $c \in \{0,1\}$ is such that $m \equiv c \pmod 2$. If $f(H_i) \ge 3 + c$ for all $1 \le i \le t$,  or $f(H_i) \le -3 - c$ for all $1 \le i \le t$, then
\[|f(K_n)| = \left|\sum_{i=1}^t f(H_i) \right| \ge t (3+c) = \frac{3+c}{m+1} {n \choose 2},\]
contradicting the hypothesis. Hence, there are indexes $i_1, i_2$ such that $f(H_{i_1}) \le 2+ c$ and $f(H_{i_2}) \ge -2- c$. Since $f(H_i) \equiv 1+m \equiv 1+c  \pmod 2$ for all $1 \le i \le t$, we deduce that, actually, 
\[f(H_{i_1}) \le 1+ c \mbox{ and } f(H_{i_2}) \ge -1- c\] 
hold. If $f(H_{i_1}) = 1+c$, then there has to be an edge $e_1 \in E(H_{i_1})$ such that $f(e_1) = 1$. Then $H_{i_1} - e_1$ is isomorphic to some member in $\mathcal{F}$ and it has $f(H_{i_1} - e_1) = f(H_{i_1}) - 1 = c$, so in this case we are done. Similarly, if $f(H_{i_2}) = -1-c$, then there has to be an edge $e_2 \in E(H_{i_2})$ such that $f(e_2) = -1$, implying that $H_{i_2} - e_2$, which is isomorphic to some member in $\mathcal{F}$, has $f(H_{i_2} - e_2) = f(H_{i_2}) + 1 = -c$ and we have finished. Hence, we can assume that 
\[ f(H_{i_1}) \le c-1 \mbox{ and } f(H_{i_2}) \ge -c+1.\] 
From this we infer the existence of edges $e'_1 \in E(H_{i_1})$ and $e'_2 \in E(H_{i_2})$  such that $f(e'_1) = -1$ and $f(e'_2) = 1$. Thus, we have that the graphs $H_{i_1} - e'_1$ and $H_{i_2} - e'_2$ are isomorphic to some members of $\mathcal{F}$ such that
\[ f(H_{i_1} - e'_1) \le c \mbox{ and } f(H_{i_2} - e'_2) \ge -c.\] 
If, in the case that $c = 1$, we have that $ f(H_{i_1} - e'_1) = 1$ or that $f(H_{i_2} - e'_2) = -1$, then we have finished. So we may assume in both cases that 
\[ f(H_{i_1} - e'_1) \le 0 \mbox{ and } f(H_{i_2} - e'_2) \ge 0.\] 
Hence, by the Interpolation Lemma (Lemma \ref{interpolation}), we obtain the existence of a zero-sum or almost zero-sum spanning subgraph $H$ which is isomorphic to some graph in $\mathcal{F}$.
\end{proof}

Observe that the condition $\min\{ e(-1), e(1) \} > \ex(n, \half(\mathcal{F}) )$ in item 1 of Theorem \ref{MasterThm} can only be satisfied if ${n \choose 2} \ge 2(\ex(n, \half(\mathcal{F}) ) + 1)$. However, by the well known observation of Erd\H{o}s that every graph $H$ contains a bipartite subgraph $H'$ with $e(H')  = \lfloor \frac{e(H)}{2} \rfloor$, $\ex( n,\half(\mathcal{F}) )  =  o(n^2)$  always holds.  Also  because of this fact, Theorem \ref{MasterThm}  could also be stated, more generally, for connected dense graphs and closed families in those connected dense graphs, for example spanning trees.  However, we will not use it in its full generality here, but a  demonstration of this  approach is given in Section 5, where also the structure of the family of dense graphs that is studied is taken heavily into account.


\section{Applications of the master theorem}
In this section, we shall demonstrate applications of the three parts of the Master Theorem concerning zero-sum spanning graphs of $K_n$. 
\subsection{Spanning paths}

 The first  application is a theorem in which we determine exactly the minimum amount of edges in each colour required in order to force a zero-sum or an almost zero-sum spanning path.  This is done using $\half(P_n)$ and a recent deep theorem of Ning and Wang \cite{NING2020111924} on Tur\'an numbers for linear forests.      

The next results are two examples of application of the Master theorem for path-decompositions and cycle-decompositions of $K_n$, which give weaker bounds than the theorem but are tailor made to demonstrate this technique.   

We need a few more definitions and results. Let $\mathcal{L}(t,k)$ denote the family of all linear forests on $t$ vertices and $k$ edges.  

\begin{theorem}[\cite{NING2020111924}] \label{thmturan}
Let $k$ and $n$ be positive integers such that $k \le n-1$. Then 
\[\ex(n, \mathcal{L}(n,k) )  = \max \left\{ \binom{k}{2},  \binom{n}{2}  -  \binom{n- \lfloor \frac{k-1}{2} \rfloor}{2} + c \right\},\] 
where $c \in \{0,1\}$ is such that $k-1 \equiv c \pmod 2$. 
\end{theorem}

In Figure \ref{diag1} below, the extremal graphs for $\ex( n, \mathcal{L}(n, k))$ given in \cite{NING2020111924} are depicted.
 
 \begin{figure}[h] 
\centering
\includegraphics{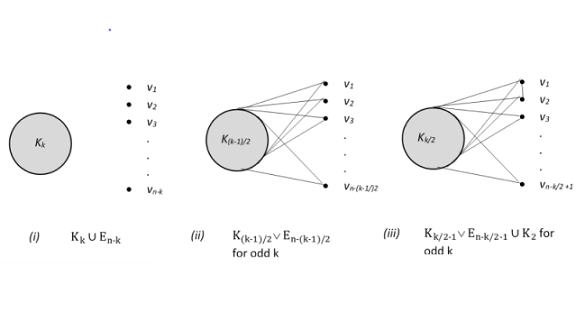} 
\caption{Extremal graphs for $\ex( n, \mathcal{L}(n, k))$ } \label{diag1}
\end{figure}

Theorem \ref{thmturan} is the key for calculating the exact Tur\'an number $\ex( n,  \half( P_ n ) )$, which we will need in order to be able to apply item 1 of the Master Theorem. Observe that Theorem \ref{thmturan} can also be stated for the family $\mathcal{L}_k$ of all linear forests with exactly $k$ edges and no isolates (recall $k \le n-1$), because we have clearly that $\ex(n, \mathcal{L}(n,k) ) = \ex(n,\mathcal{L}_k )$. 
\begin{theorem} \label{thmspaths}
Let $n \geq 3$, let $c \in \{0,1\}$ be such that $\lfloor \frac{n-1}{2} \rfloor - 1 \equiv c \pmod 2$, and let $f : E(K_n)  \rightarrow \{ -1 ,1\}$ be a colouring fulfilling  
\[\min \{ e( -1) ,e(1) \} >  \binom{n}{2}  -   \binom{n - \lfloor \frac{ n-3}{4} \rfloor}{2} + c.\] 
Then  there is a zero-sum or an almost zero-sum spanning path.  Moreover, the bound is sharp.
\end{theorem}
\begin{proof}
As mentioned before, the family of all spanning paths of $K_n$ is a closed family. Moreover, it is straightforward to see that $\half(P_ n) = \mathcal{L}_{\left\lfloor \frac{n-1}{2} \right\rfloor}$. Hence, by Theorem \ref{thmturan}, we have 
\[\ex( n,  \half( P_ n ) ) =  \ex \left( t , \mathcal{L}_{\left\lfloor \frac{n-1}{2} \right\rfloor} \right)  = \binom{n}{2}  -   \binom{n - \lfloor \frac{ n-3}{4} \rfloor}{2} + c.\]
Together with the simple inequalities $\lfloor \frac{n-3}{4}\rfloor \le \frac{n-3}{4}$ and $c \le 1$, it is straightforward to check that $2(\ex( n,  \half( P_ n ) )+1) \le \binom{n}{2}$ for $n \ge 4$. For $n = 3$, the same inequality can be checked separately.
Hence, together with the hypothesis that $\min \{ e( -1) ,e(1) \} >  \binom{n}{2}  -   \binom{n - \lfloor \frac{ n-3}{4} \rfloor}{2} + c$, it follows by item 1 of Theorem \ref{MasterThm} that there is a zero-sum or an almost zero-sum spanning path and that the bound is sharp.
 \end{proof}
The sharpness of this theorem is a consequence of the Master Theorem coming from the extremal examples of $\mathcal{L}\left(n, \lfloor \frac{n-1}{2} \rfloor \right)$-free graphs on $n$ vertices, which were given in \cite{NING2020111924} (see Figure \ref{diag1}). Thus, a colouring $f: E(K_n) \to \{-1,1\}$ where the $(-1)$-edges (or, equivalently, the $1$-edges) induce one of these extremal graphs is a colouring in which the number of edges in one of the colours has one unit less than what is allowed in Theorem \ref{thmspaths} but where no zero-sum or almost zero-sum spanning path can be found.
%

 
Now we give an example of application of the Master Theorem in the case of decomposition into a closed family of Hamiltonian paths.

 \begin{example} \label{example1}
Let $n \equiv 0 \pmod 2$ and let $f : E(K_n) \rightarrow \{ -1 ,1\}$  such that  $|f(K_n)| < \frac{3n}{2}$.  Then  there is a spanning path $Z$ with $|f(Z)| = 1$.   
\end{example}
 
\begin{proof}
Since  $n \equiv 0 \pmod 2$, $K_n$ can be decomposed into $\frac{n}{2}$ Hamilton paths (see \cite{Harary1969}), hence into $\frac{n}{2}$  paths on $n-1$ edges. Then, with $n -1 \equiv 1 = c \pmod 2$, and the hypothesis that \[|f(K_n)| < \frac{3n}{2} = \frac{2+c}{n-1} \binom{n}{2},\]
item 2 of Theorem \ref{MasterThm} yields that there is an almost spanning path.   



\end{proof}

Now we give an example of using the Master theorem in the case of decomposition into Hamiltonian cycles which is the covering family of the closed family of Hamiltonian paths. 

\begin{example}
Let $n \equiv 1 \pmod 2$ and let $f : E(K_n)  \rightarrow \{ -1 ,1\}$  such that  $|f(K_n)| < \frac{3(n-1)}{2}$.  Then  there is a zero-sum spanning path.  
\end{example}

\begin{proof}
Since  $n \equiv 1 \pmod 2$, $K_n$ can be decomposed into $\frac{n-1}{2}$ Hamiltonian cycles (see \cite{Harary1969}), hence into $\frac{n-1}{2}$  cycles on $n$ edges. Then, with $n -1 \equiv 0 = c \pmod 2$, and the hypothesis that \[|f(K_n)| < \frac{3(n-1)}{2} = \frac{3+c}{n} \binom{n}{2},\]
item 3 of Theorem \ref{MasterThm} yields that there is a zero-sum spanning path.





 \end{proof}

\subsection{Spanning trees}
Let $\mathcal{F}_k$ be the family of forests on $k$ edges without isolated vertices.
\begin{lemma} \label{lemma1}
For integers $n, k$ such that $n \ge k$, ${\rm ex}(n, \mathcal{F}_k) = \binom{k}{2}$.
\end{lemma}
 
\begin{proof}
We prove this by induction on $k$.  For $k = 1,2$  it is true and we assume this is true for $k$, so let us prove it for  $k+1$. 

Let $G$ be a graph of order $n \ge k+1$ and with $e(G) > \binom{k +1}{2}$.  If $\Delta(G)  \geq k +1$ we are done as there is a star on at least $k+1$  edges. So we assume that $\Delta(G) \leq k$.  Let $v$ be a vertex with $\deg(v)  \geq 1$.   Delete $v$ to get $G^* = G - v$.  Clearly, $e(G^*) > \binom{k+1}{2}  -  k \geq \binom{k}{2}$. Hence, by the induction hypothesis, $G^*$ contains a forest $F^*$  on $k$ edges. Adding the vertex $v$ and a single edge incident with $v$ to $F^*$, we get a forest $F$ with $e(F) =  k+1$.

The graph $K_k \cup \overline{K_{n-k}}$ shows that the bound is sharp.
\end{proof}

\begin{theorem}
Let  $f : E(K_n)  \rightarrow \{ -1 ,1\}$ be a colouring with  $\min \{ e(-1), e(1)  \} > \binom{\lfloor \frac{n-1}{2} \rfloor}{2}$. Then $K_n$ contains a spanning zero-sum or an almost zero-sum tree, and this result is sharp.
 \end{theorem}

\begin{proof}
Let $\mathcal{T}_n$ be the family of all spanning trees of $K_n$, i.e. the family of all trees on $n-1$ edges, which is a closed family by Lemma \ref{matroids}. Clearly, ${\rm Half}(\mathcal{T}_n) = \mathcal{F}_{\lfloor \frac{n-1}{2} \rfloor}$. Then we have, with Lemma \ref{lemma1}, that
\[\min \{ e(-1), e(1)  \} > \binom{\lfloor \frac{n-1}{2} \rfloor}{2} = {\rm ex}(n, \mathcal{F}_{\lfloor \frac{n-1}{2} \rfloor}) = {\rm ex}(n, {\rm Half}(\mathcal{T}_n)).\] 
Hence, Theorem \ref{MasterThm} yields the result.

The only requirement now for this to work is that  $2 \left( \binom{\lfloor \frac{n-1}{2} \rfloor}{2} +1 \right) \leq \binom{n}{2} = e(K_n)$, which is always true.
\end{proof}

\subsection{Zero-sum trees of diameter at most $3$}

We now consider spanning trees of diameter at most three. We denote by $\mathcal{T}_n^3$ the family of all spanning trees of $K_n$ of diameter at most $3$.

\begin{lemma} \label{lemmadiam3}
$\mathcal{T}_n^3$ is a closed family. Moreover, for any two spanning trees $T, T^* \in \mathcal{T}_n^3$, there is a chain $T = T_1, T_2, \ldots, T_q = T^*$ of trees contained in $\mathcal{T}_n^3$ such that, for $1 \le i \le q-1$, $T_{i+1}$ is obtained from $T_i$ by an edge-replacement and such that $q \le 2(n-2)$.
\end{lemma}

\begin{proof}
We observe that a spanning tree of $K_n$ of diameter at most $3$ is either a star $K_{1,n-1}$ or a double star  $S_{p,q}$ with centres $u$ and $v$ adjacent, such that $u$ has $p$ leaves, $v$ has $q$ leaves, and such that $p+q  =  n-2$.  We look at the following operations of edge-replacements.

\begin{enumerate}
\item{\emph{A spanning tree of $K_n$ of diameter $3$ which is a double star $S_{p,q}$ can be transformed into a spanning star via spanning trees of diameter at most $3$. This process requires $q$ steps.}

Let the spanning tree  be $S_{p,q}$ with $u$ and $v$ the centres, where $u$ has $p$ leaves and.  Consecutively, for every leaf $z$ incident with $v$,  delete the edge $zv$ and add the edge $zu$ until all leaves of $v$ are attached to $u$, giving a  spanning star with centre $u$.  All the spanning trees in the process  have diameter $3$ except  the final star which has diameter $2$.}

\item{\emph{Any spanning star with centre $u$ can be transformed into a spanning star with centre $v \in V(K_n) \setminus \{u\}$ with all the trees in the process having diameter $3$. This process requires $n-2$ steps.}

Consecutively, for every leaf $z$ adjacent to the centre $u$ with exception of $v$, delete the edge $zu$ and add the edge $zv$ until all leaves are done.  All the intermediate trees in the process are of diameter $3$. }

\item{\emph{Any spanning star of $K_n$ with centre $u$ and leaf $v$ can be transformed into a double star $S_{p,q}$  with centres $u$ and $v$ having a particular set of $p$ leaves attached to $u$ and the remaining $q = n-2 - p$ leaves attached to $v$. This process requires $q$ steps.}

Given a spanning star with centre $u$ and a leaf $v$, and a set $S$ of $p$ vertices to remain leaves attached to $u$,  for every leaf $z$ of $u$ not in $S$, we delete the edge $zu$ and add the edge $zv$, giving the required double star $S_{p,q}$.  With exception of the star with which we started, every step involves a spanning tree of diameter $3$.}
\end{enumerate}

These three operations suffice to transform any spanning tree of $K_n$ of diameter at most $3$ to any other spanning tree of diameter at most $3$ with all trees in the process having diameter at most $3$.  In fact, the number of edge replacements is bounded by $2(n-2)$ edge-replacements. Indeed, if we have any pair of spanning trees $T$  and $T^*$ of diameter at most $3$, we have the following possible situations. If $T$ is a double star and $T^*$ a  star or vice versa, we require at most $n-3$ steps by operation 1 or 3 above. If $T$ and $T^*$ are both stars, we require at most $n-2$ steps by operation 2. Finally, if $T$ and $T^*$ are both double stars, we can use operation 1, then operation 2 and then operation 3, but we can do it the most efficient as possible. Suppose $T \cong S_{p,q}$ with $p\leq q$ and centres $u$ and $v$, and $T^* \cong S_{r,s}$ with $r \leq s$ and centres $x$ and $y$. We transform $T$ to a star with centre $v$, which involves $p \le \frac{n-2}{2}$ steps.  Then transform this star with centre $v$ to a star with centre $y$, involving $n -2$ steps.  Finally, we transform this star to  $T^*$ with $x$ as the centre, involving another $r \le \frac{n-2}{2}$ steps. Hence, the number of edge-replacements is at most $\frac{n-2}{2} +(n-2) +\frac{n-2}{2}  = 2(n-2)$.
\end{proof}

It is well known that ${\rm ex}(n,K_{1,k}) = \left \lfloor \frac{k-1}{2} n \right\rfloor$, see for instance \cite{LLP2012}. We will use this result for the following theorem.

\begin{theorem}
Let $n \geq 3$ and let $f: E(K_n) \rightarrow \{ -1 ,1\}$ such that 
\[\min\{ e( -1)  , e( 1) \}  > \left\lfloor \frac{n}{2} \left\lfloor\frac{n-3}{2}\right\rfloor \right\rfloor.\] 
Then there is a zero-sum or an almost zero-sum spanning tree of diameter at most $3$. 
\end{theorem}

\begin{proof}
Since $K_{1, n-1} \in \mathcal{T}_n^3$, we clearly have $K_{1, \lfloor\frac{n-1}{2} \rfloor} \in {\rm Half}(\mathcal{T}_n^3)$. Hence, we can conclude that
\[{\rm ex}(n, {\rm Half}(\mathcal{T}_n^3)) \le {\rm ex}(n, K_{1,\lfloor\frac{n-1}{2} \rfloor}) = \left \lfloor \frac{\lfloor\frac{n-1}{2} \rfloor-1}{2} n \right\rfloor =\left\lfloor \frac{n}{2} \left\lfloor\frac{n-3}{2}\right\rfloor \right\rfloor.\]
By Lemma \ref{lemmadiam3},  $\mathcal{T}_n^3$ is a closed family. Hence, by Theorem \ref{MasterThm}, there is a zero-sum or an almost zero-sum spanning tree of diameter at most $3$.

It remains to check that  $2 \left(\left\lfloor\frac{n}{2} \left\lfloor\frac{n-3}{2}\right\rfloor\right\rfloor +1 \right) \leq \frac{n(n-1)}{2}$,  which holds true for $n \geq 3$.
\end{proof}

Observe that one cannot hope to obtain a zero-sum or an almost zero-sum star (i.e. a tree of diameter $2$) even in the case that $\{e(-1), e(1) \} = \left\{ \left\lceil\frac{(n-1)}{4}\right\rceil, \left\lfloor\frac{(n-1)}{4}\right\rfloor \right\}$ where  $|f(K_n)| \le 1$.  This is because, for infinitely many $n$'s there
 are integers $x$ and $y$ such that  $x +y  = n$, $x >y$, for which we can split $V(K_n) = (X \cup  Y)$  such that  $|X|= x$,  $|Y| = y$  and  $\binom{x}{2}  =  xy + \binom{y}{2}$, being able to colour all edges in $X$ with $-1$ and  all edges in $E(K_n) \backslash E(X)$  with  $+1$ such that $e(-1) = e(1)$ \cite{caro2019zero2, clztotomni}, but where there is no zero-sum spanning star nor an almost zero-sum spanning star.

\section{Zero-sum spanning trees in other graph classes}

In this section, we will deal with graphs classes that are different from the family of complete graphs.

\subsection{Zero-sum spanning trees  in $K_3$-free graphs}

\begin{lemma} \label{lemmabip}
Let $k$ be a positive integer, and let $G$ be a $K_3$-free graph with $e(G) > \lfloor \frac{k^2}{4}\rfloor$. Then $G$ contains a forest $F$ on at least $k$ edges, and the bound is sharp.  
\end{lemma}

\begin{proof}
We proceed by induction on $k$.  For $k = 1,2$, it is true.  Assume this is true for $k$ and let us prove it for $k+1$.
Observe first that $\lfloor \frac{q^2}{4} \rfloor =  \lfloor \frac{q}{2} \rfloor \lceil \frac{q}{2} \rceil$ for any non-negative integer $q$. Now let $G$ be a $K_3$-free graph with $e(G) > \lfloor \frac{(k+1)^2}{4}\rfloor $. We may assume no vertex in $G$ is isolated.  Suppose first that $\delta(G) \leq \lfloor \frac{k+1}{2}\rfloor$ and let $v$ be a vertex of minimum degree in $G$.  Consider $G^* = G -  v$.  Then $G^*$ is again $K_3$-free, and has the following edge-number:
\begin{align*}
e(G^*) = e(G) - \deg_G(v) & > \left\lfloor \frac{(k+1)^2}{4} \right\rfloor -  \left\lfloor \frac{k+1}{2}\right\rfloor \\
& =\left\lfloor \frac{k+1}{2} \right\rfloor \left\lceil \frac{k+1}{2} \right\rceil -  \left\lfloor \frac{k+1}{2}\right\rfloor  \\
& = \left\lfloor \frac{k+1}{2} \right\rfloor \left(\left\lceil \frac{k+1}{2} \right\rceil -  1 \right) \\
& = \left\lfloor \frac{k}{2} \right\rfloor \left\lceil \frac{k}{2} \right\rceil = \left\lfloor \frac{k^2}{4} \right\rfloor .
\end{align*}
Hence, by the induction hypothesis, $G^*$ contains a forest $F^*$ on $k$ edges.  Now add $v$ and one edge incident with $v$ to get a forest $F$ with $k+1$ edges, as required.

So we may assume $\delta(G) \geq \lfloor \frac{k+3}{2} \rfloor$. Take two adjacent vertices $u$ and $v$.  Since $G$ is $K_3$-free,  $u$ and $v$ have no common neighbour.  So all the edges incident with $v$ but not $u$, together with all the edges incident with $u$ but not $v$, and the edge $uv$ form a tree on at least $2 \lfloor\frac{k +1}{2}\rfloor +1  \ge   k +1$ edges and we are done.

For the sharpness, take the complete bipartite graph $K_{\lfloor \frac{k}{2}\rfloor, \lceil \frac{k}{2}\rceil}$ which has $ \lfloor \frac{k}{2}\rfloor \lceil \frac{k}{2}\rceil$ edges but no forest on $k$ edges.
\end{proof}
 
\begin{theorem}
Let $n$ be a positive integer and let $G$ be a connected $K_3$-free graph of order $|V(G)| \in \{2n, 2n+1\}$ such that $e(G) \ge \lfloor \frac{n^2}{2} \rfloor+2$. Let $f: E(G)  \rightarrow \{ -1 ,1\}$ be such that  $\min \{ e(-1), e(1)  \} > \lfloor \frac{n^2}{4}\rfloor$.  Then  there is a spanning zero-sum or an almost zero-sum tree, and these bounds are sharp.
\end{theorem}

\begin{proof}
Consider the graph $G^-$ induced by the $-1$-edges, and the graph  $G^+$ induced by the $1$-edges. Since $\min\{ e(G^-), e(G^+)\} > \lfloor \frac{n^2}{4}\rfloor$, we can use Lemma \ref{lemmabip} with $k = n$, from which follows that $G^-$ and $G^+$ contain a forest $F^-$ and, respectively, $F^+$ on at least $n$ edges each.  Complete $F^-$  to a spanning tree $T^-$ on $2n$ edges and $F^+$ to a spanning tree $T^+$ on $2n$ edges. Since, clearly, $f(F^-) \le 0$ and $f(F^+) \ge 0$, the Interpolation Lemma \ref{interpolation} yields the result.

We now consider sharpness. Let $G$ be a connected bipartite graph on $2n$ or $2n+1$ vertices and at least $\lfloor \frac{n^2}{2} \rfloor+2$ edges such that $G$ has a subgraph $H$ isomorphic to $K_{\lfloor \frac{n}{2} \rfloor,\lceil \frac{n}{2} \rceil}$. Evidently, $G$ is $K_3$-free. We colour all edges from $H$ with $-1$, and all remaining edges with $1$.  Clearly, $e(1)  > e( -1)  = \lfloor \frac{n}{2} \rfloor \lceil \frac{n}{2} \rceil =  \lfloor \frac{n^2}{4} \rfloor$. To have a zero-sum spanning tree, we need $n$ edges coloured $-1$  and $n$  edges coloured  $+1$.  However,  from $E(H)$ we can take only $n-1$ edges, since otherwise we would have a cycle.
\end{proof}

\subsection{Zero-sum spanning trees in $d$-trees}

For an integer $d \geq 1$, a graph $G$ is said to be  \emph{$d$-degenerate} (also called in the literature \emph{partial $d$-tree}) if for every induced subgraph $H$ of $G$, $\delta(H) \leq d$. A \emph{$d$-tree} is a maximal $d$-degenerate graph, that is, a graph obtained from the complete graph $K_{d+1}$  by successively adding vertices, each vertex being adjacent to exactly $d$ vertices in the former graph.  The number of edges in a $d$-tree on $n \geq d +1$  vertices is $nd  - \binom{d +1}{2}$.  It  is clear  that every $d$-tree on $n$  vertices contains an induced $d$-tree on $n'$ vertices for every k,  $d +1 \leq n' \leq n$.  For early surveys on $d$-degenerate graphs and $d$-trees, see  \cite{lick1970k, Rose1974}, and further results on maximal $d$-degenerate graphs and $d$-trees can be found in \cite{filakova1997note, Patil1986}.

 \begin{lemma} \label{degenerate}
Let $k$ and $d$ be non-negative integers.  Let $G$ be a $d$-degenerate graph with
\[e(G) > \left\{\begin{array}{ll}
\binom{k}{2} , & \mbox{if } k \le d,\\
kd - \binom{ d+1}{2}, & \mbox{else}.
\end{array} \right.\]
Then $G$ contains a forest on at least $k$ edges. Moreover, the bound is best possible for any choice of the parameters.
 \end{lemma}

\begin{proof}
If $d  = 1$, then $G$ is itself a forest and the bound on $e(G)$ is equal to $k-1$ in both cases, so the claim is trivial.
So we assume  $d \geq 2$.  For the case that $k \leq d$, the claim follows from Lemma \ref{lemma1}.  So we may assume that $k \ge d+1$. We will prove the statement by induction on $k$. Observe that, for $k = d +1$, $kd - \binom{ d+1}{2} = \binom{k}{2}$, so the base case is also covered by Lemma \ref{lemma1}. Suppose now the claim is true for $k$.  Now let $G$ be a $d$-degenerate graph with $e(G) > ( k +1)d - \binom{d +1}{2}$ edges. We can assume, without loss of generality, that $G$ has no isolated vertices.  Let $v \in V(G)$ be a vertex of minimum degree and define $G^* = G-v$, which is again $d$-degenerate. By the $d$-degeneracy of $G$, we have $1 \leq \deg(v) \leq d$. It follows that
\begin{align*}
e(G^*) & >  (k +1)d - \binom{d+1}{2} - \deg( v)  \\
& \geq ( k +1)d - \binom{d+1}{2}  - d \\
&=  kd - \binom{d+1}{2}.
\end{align*} 
Hence, by the induction hypothesis, $G^*$ contains a forest $F^*$ on $k$ edges.  Adding  the vertex $v$ and exactly one edge incident with it to $F^*$, we get a forest $F$ with at least $k+1$ edges in $G$, proving the induction step.

For the sharpness, consider, for the case that $1 \leq  k \leq d +1$, the complete graph $K_k$. Indeed, $K_k$ is $d$-degenerate and  has  $\binom{k}{2}$  edges, but no forest on $k$ edges. When $k \geq  d +1$, take any $d$-tree on $k$ vertices with exactly $kd - \binom{d+1}{2}$ edges, which contains no forest on $k$ edges.
\end{proof}

 We can now prove the theorem about zero-sum  spanning trees in $d$-trees.

\begin{theorem} \label{themdegenerate}
Let $G$ be a $d$-tree on $n \geq 2d+2$ vertices.  Let  $f: E(G)  \rightarrow \{ -1 ,1\}$  such that  $\min \{ e(-1), e(1)  \} > \lfloor\frac{n-1}{2}\rfloor d -\binom{d+1}{2}$.  Then  $G$ contains a  zero-sum or an almost zero-sum spanning tree, and the bound is sharp.
\end{theorem}

\begin{proof}
Observe first that the condition $\min \{ e(-1), e(1)  \} > \lfloor\frac{n-1}{2}\rfloor d -\binom{d+1}{2}$ can always be satisfied. Indeed, for this to hold, it is required that  
\[2 \left(\left\lfloor\frac{n-1}{2} \right\rfloor d -\binom{d+1}{2}+1 \right) \leq nd - \binom{d+1}{2} = e(G),\]
which is always true. To prove the statement, we use Lemma \ref{degenerate} with $k = \lfloor\frac{n-1}{2}\rfloor \ge d+1$.  Then there exists a forest $F^-$ on at least $\lfloor\frac{n-1}{2}\rfloor$ edges coloured $-1$ and a forest $F^+$ on at least $\lfloor\frac{n-1}{2}\rfloor$ edges coloured $+1$.  Complete $F^-$  to a spanning tree $T^-$ on  $n-1$  edges  and complete $F^+$ to a spanning tree $T^+$ on $n-1$ edges. Then we have that $f(T^-) \leq 1$  while $f(T^+) \geq -1$.  If $f(T^-) = 1$  or $f(T^+) = -1$, $T^-$  or $T^+$ is an almost zero-sum tree. Hence, we can assume that $f(T^-) \leq 0$ and $f(T^+) \geq 0$, and so it follows with Lemma \ref{interpolation} that there is a  zero-sum or an almost zero-sum spanning tree. For the sharpness, consider any $d$-tree on $n \geq 2d+2$ vertices. Then it contains a $d$-tree $H$ on $\lfloor\frac{n-1}{2}\rfloor \geq  d + 1$ vertices which has clearly $e(H)  =  \lfloor\frac{n-1}{2}\rfloor d - \binom{d+1}{2}$.  Let  $f: E(G)  \rightarrow \{ -1 ,1\}$  such that  the edges of $H$ are coloured $ - 1$, and all remaining edges are coloured $1$.   Observe that $e(1) > e(-1)  =  \lfloor\frac{n-1}{2}\rfloor d -\binom{d+1}{2}$.  To  have a zero-sum spanning tree of $G$, we need precisely a forest with $\lfloor\frac{n-1}{2}\rfloor$ edges coloured  $- 1$,  which we are forced to choose from $H$. However, this is impossible since $|V(H)| = \lfloor\frac{n-1}{2}\rfloor$.
\end{proof}

\subsection{Zero-sum spanning trees in maximal planar graphs}

Since planar graphs are $5$-degenerate, results like Lemma \ref{degenerate} with $d=5$ and its consequences hold for planar graphs. However, one would expect that better bounds can be obtained than in the previous section by exploiting the special properties of planarity. We do this here for maximal planar graphs, keeping in mind that a maximal planar graph on $n \geq 3$ vertices has $3n-6$ edges and that any embedding of such a graph in the plane has the outer face bounded by a triangle.

\begin{lemma}\label{la:planar_forest}
Let  $k \ge 3$ be an integer and let $G$ be a planar  graph with $e(G) \ge 3k-5$. Then $G$ contains a forest on at least $k$ edges and the bound is best possible.
 \end{lemma}
 
\begin{proof}
 We will prove the statement by induction on $k$. If $k = 3$, then $3k-5 = 4 > 3 = \binom{k}{2}$. Hence, the induction start follows by Lemma \ref{lemma1}. We assume now the bound holds for $k$, where $k \ge 3$, and we will prove it for $k+1$. Let $G$ be  a planar graph with at least $3(k+1)  - 5$ edges and order $n$. We may assume that $G$ has no isolated vertices. Let $G_1,\ldots, G_t$  be the connected components of $G$ of order $n_1,\ldots,n_t$, where $t \geq 1$. If $G$ has a component, say $G_t$, isomorphic to $K_2$, then consider the planar graph $G^* = G - G_t$, and observe that
 \[e(G^*) = e(G) -1 \ge 3(k+1) - 5 -1 = 3k-3 > 3k - 5.\]
Hence, by the induction hypothesis, $G^*$ contains a forest $F$ on at least $k$ edges, and thus $F \cup G_t$ is a forest contained in $G$ that has at least $k+1$ edges, so we are done. Therefore, we may assume that $n_i \ge 3$ for all $1 \le i \le t$. Since every component in $G$ is planar, we have
\[3n-6t = \sum_{i=1}^t (3n_i-6) \ge \sum_{i=1}^t e(G_i) = e(G) \ge 3(k+1)-5 = 3k - 2.\]
This gives $n \ge k+2t - \frac{2}{3}$, so in fact $n \ge k+2t$. On the other hand, every component $G_i$ of $G$ has a spanning tree $T_i$ on $n_i - 1$ edges, which together produce a spanning forest $F = \cup_{i=1}^t T_i$ with
\[e(F) = \sum_{i=1}^t (n_i - 1) = n - t \ge k+t \ge k+1,\]
and we are done. \\
For the sharpness, consider any maximal planar graph on $k$ vertices, which has exactly $3k-6$ edges, but no forest on $k$ edges.
\end{proof}

We can now prove the theorem about zero-sum or almost zero-sum spanning trees in maximal planar graphs.

\begin{theorem}
Let $G$ be a maximal planar graph on $n \geq 7$ vertices.  Let  $f: E(G)  \rightarrow \{ -1 ,1\}$  such that  $\min \{ e(-1), e(1)  \} \geq 3 \lfloor \frac{n-1}{2} \rfloor -5$.  Then  $G$ contains a  zero-sum or almost zero-sum spanning tree, and the bound is sharp.
\end{theorem}

\begin{proof}  
Observe that the condition $\min \{ e(-1), e(1)  \} \geq 3 \lfloor \frac{n-1}{2} \rfloor -5$ can always be satisfied since $2(3\lfloor \frac{n-1}{2} \rfloor -5) \leq 3n-6 = e(G)$.
To show the existence of the desired spanning tree, we use Lemma \ref{la:planar_forest} with $k = \lfloor \frac{n-1}{2} \rfloor \geq  3$.  Since $\min \{ e(-1), e(1)  \} \geq 3 \lfloor \frac{n-1}{2} \rfloor -5$, there exist forests  $F^-$ and $F^+$, each on at least $\lfloor \frac{(n-1)}{2} \rfloor$ edges such that all edges of $F^-$ are coloured $-1$ and all edges of $F^-$ are coloured $1$. Complete $F^-$ to a spanning tree $T^-$ on  $n-1$  edges and complete $F^+$ to a spanning tree $T^+$ on $n-1$ edges. Then we have that $f(T^-) \leq 1$, while $f(T^+) \geq -1$.  If $f(T^-) = 1$  or $f(T^+) = -1$, $T^-$  or $T^+$ is an almost zero-sum tree. Hence, we can assume that $f(T^-) \leq 0$ and $f(T^+) \geq 0$, and so it follows with Lemma \ref{interpolation} that there is a  zero-sum or almost zero-sum spanning tree.

To show the sharpness, we proceed in an analogous way to the proof of sharpness for Theorem \ref{themdegenerate}. We start with $K_3$ embedded in the plane and add, recursively, new vertices such that, each time a new vertex is added, it is made adjacent to the three vertices bounding the current outer face. In this way, we can obtain, for any $3 \leq k \leq n$, a maximal planar graph of order $n$ containing an induced maximal planar subgraph of order $k$. So, for $n\geq 7$, let $G$ be a maximal planar graph of order $n$ containing a maximal planar subgraph $P$ on $\lfloor\frac{n-1}{2} \rfloor$ vertices.  Let  $f: E(G)  \rightarrow \{ -1 ,1\}$ such that  the edges of $P$ are coloured $-1$ and the remaining edges are all coloured $1$.  Observe that $e(1) >  e(-1) =  \lfloor \frac{n-1}{2} \rfloor-6$.  To  have a zero-sum spanning tree of $G$, we need precisely a forest of $\lfloor \frac{n-1}{2} \rfloor$ edges coloured  $-1$,  which we have to take from $P$, and this is impossible since $|V(P)| = \lfloor \frac{n-1}{2} \rfloor$.
\end{proof}

\section{Zero-sum connectivity}
 
\begin{theorem}\label{thm:connected}
Let $n\geq 6$ and $f:E(K_n)\to \{-1,1\}$ such that  $\min\{e(-1),e(1)\}\geq \left\lceil\frac{n+1}{2}\right\rceil$. Then for every two vertices $x$ and $y$, there is a zero-sum path of length at most 4 with $x$ and $y$ the end vertices of this path.  Furthermore the lower bound is sharp.

\end{theorem}

\begin{proof}
For every pair of vertices $x$ and $y$ having another vertex $u$ such that $f(ux)\neq f(uy)$ we are done, as $x$-$u$-$y$ is a zero-sum path. So let $x,y$ be a pair of vertices such that for every $u\in V\setminus \{x,y\}$, $f(xu)=f(uy)$. This splits $V\setminus\{x,y\}$ into two sets $A$ and $B$ (one possibly empty), where $A:=\{u:f(xu)=f(uy)=1\}$ and $B:=\{u:f(xu)=f(uy)=-1\}$.
Assume, without lost of generality, that $|A|\geq |B|$ (otherwise we can multiply all colours by $-1$ and do as follows).\\

\noindent
\emph{Case 1:} If $|B|=0$ we have $|A|=n-2\geq 4$. By hypothesis $e(-1)\geq \left\lceil\frac{n+1}{2}\right\rceil$, so even if $f(xy)=-1$ we still have $\left\lceil\frac{n+1}{2}\right\rceil-1$ edges coloured $-1$  in the subgraph induced by $A$. Since $\left\lceil\frac{n+1}{2}\right\rceil-1\geq \left\lceil\frac{n-2}{2}\right\rceil$, this is enough to guarantee the existence of three vertices $u,v,w\in A$ such that $f(uv)=f(vw)=-1$. Then, $x$-$u$-$v$-$w$-$y$ is a zero-sum path. \\

\noindent
\emph{Case 2:}  If $|B|=1$ we have $|A|=n-3\geq 3$. Let $z$ be the only vertex in $B$. If there are vertices $u,v \in A$ such that $f(zu)=f(zv)=-1$ then $x$-$u$-$z$-$v$-$y$ is a zero-sum path.  Suppose there is only one vertex $u\in A$ such that $f(zu)=-1$ then, either there is a vertex $v\in A$ such that $f(vu)=1$, or there are two vertices $v,w\in A$ such that $f(vu)=f(wv)=-1$; in the first case $x$-$v$-$u$-$z$-$y$ is a zero-sum path, while in the second case $x$-$v$-$u$-$w$-$y$ is a zero-sum path. It remains to consider the case where $f(zu)=1$ for every $u\in A$; in this case note that outside $A$ there are at most three edges coloured $-1$ and, since $ \left\lceil\frac{n+1}{2}\right\rceil\geq 4$, we must have an edge $uv$, with $u,v\in A$, such that $f(uv)=-1$, hence $x$-$u$-$v$-$z$-$y$ is a zero-sum path.\\

\noindent
\emph{Case 3:} If $|B|\geq 2$ we have $|A|=n-4\geq 2$. Consider $u$, $v$, $z$, $w$ vertices such that $u, v\in A$ and $z, w\in B$. If $f(uz)=f(uw)=1$ then $x$-$z$-$u$-$w$-$y$ is a zero-sum path. So, we may assume without lost of generality that $f(uz)=-1$. Then $f(vu)=-1$ (otherwise $x$-$v$-$u$-$z$-$y$ would be a zero-sum path), and $f(vz)=1$ (otherwise $x$-$v$-$z$-$u$-$y$ would be a zero-sum path) but then $x$-$z$-$v$-$u$-$y$ is a zero-sum path.\\

\medskip

\noindent  In order to show that the lower bound in Theorem \ref{thm:connected} is best possible we exhibit, for $n=4$ and $n=5$, a colouring function $f^*:E(K_n)\to \{-1,1\}$ with $\min\{e(-1),e(1)\}=\left\lceil\frac{n+1}{2}\right\rceil$ and two vertices such that there is no zero-sum path of length at most 4 between them; and, for $n\geq 6$,  a colouring $f^*:E(K_n)\to \{-1,1\}$ with $\min\{e(-1),e(1)\}=\left\lceil\frac{n+1}{2}\right\rceil-1$ and a pair of vertices with the same property.

Let $n\in \{4,5\}$, and let $f^*:E(K_n)\to \{-1,1\}$ be a colouring such that the $-1$-edges induce a $K_3$, being $a,b,c$ its $3$ vertices. Then $e(-1)=3=\left\lceil\frac{n+1}{2}\right\rceil$, and there is no zero-sum path of length at most $4$ with $a$ and $b$ being the end vertices of this path.

Let $n\geq 6$. Consider a colouring $f^*:E(K_n)\to \{-1,1\}$ such that the set of edges coloured $-1$ induce a $\lfloor \frac{n}{2} \rfloor K_2$. Hence, $e(-1)= \lfloor \frac{n}{2} \rfloor = \left\lceil\frac{n+1}{2}\right\rceil-1$, and for two vertices  $a$ and $b$ such that $f^*(ab)=-1$ there is no zero-sum path of length at most 4 between them; to see this, note that such a path has to start and end with edges coloured $1$, thus the middle edges have to be both coloured $-1$, which is impossible.
\end{proof}

Next we show that the lower bound is in fact critical to zero-sum connectivity without any restriction on the length of the zero-sum paths.

With $\min \{e(-1), e(1)\} \leq  \lceil \frac{n-1}{2} \rceil$, there is a colouring preventing zero-sum path of any length (not only 2 or 4)  --- we take a matching of  $\lceil \frac{n-1}{2} \rceil$ edges coloured $-1$ and the rest coloured $+1$,  so that no two vertices adjacent by an edge coloured $-1$ have a zero-sum path between them.

Next we show that distance four zero-sum paths are inevitable.  We cannot force every two vertices to have a zero-sum path of length two even if $\min\{e(-1), e(1)\}  = \lfloor \frac{n(n-1)}{4} \rfloor$ (half the edges).  Consider $K_n$  (n sufficiently large) and choose two vertices $x$ and $y$.  Connect $x$ and $y$ to all other vertices by edges coloured  $-1$.  The rest of the edges  are coloured $-1$ and  $+1$, just to get the number of $-1$ and $+1$ edges  as equal as possible.  Every $(x,z,y)$-path has weight $-2$ and $|e(-1)-e(1)| \leq 1$.  

So $n$ sufficiently large is just to make sure the first $2(n-2)$ edges coloured $-1$  are at most  half the number of edges,  namely  $2(n-2)  \leq \frac{n(n-1)}{4}$ which is  true for $n \geq 7$.

\section{Conclusion}

We have studied, for various graph families $\cal F$,  the problem of finding conditions on a 2-colouring of $E(H)$ ($H$ is, in most cases, the complete graph $K_n$) with colours $-1$ and $+1$ such that, given a graph $G$ in $\cal F$, there is, in any such 2-colouring of $H$, a copy of $G$ such that the sum of the colours on $E(G)$ is zero or $\pm 1$. Usually the conditions on the colouring take the form of bounds on the number of edges coloured $-1$ or $+1$.  Most of the bounds we have obtained are sharp, but not all, for example in the case  of spanning trees of diameter at most three.

We have given a unified treatment of this problem by obtaining most of our results as consequences of a Master Theorem which we have applied for classes $\cal F$ such as spanning subtrees or spanning paths. It would be interesting to be able to do the same for some  families of dense graphs. 

Another intriguing problem is to find conditions such that any $\pm1$-colouring of $K_{4n}$ has a zero-sum matching on $2n$ edges. The main difficulty in applying our techniques is that  the graph $2nK_2$  is not a local  amoeba in $K_{4n}$, hence the various copies of matchings in $K_{4n}$ is not a  closed family and we cannot use Lemma \ref{interpolation}  and Theorem \ref{MasterThm}.

The following construction shows that, for infinitely many integers $n$, there are $\{-1,1\}$-colourings of the complete graph $K_{4n}$, where the number of edges of both colours is nearly balanced, but such a colouring does not contain a zero-sum spanning matching. To see this, let $n$ be a square, say $n = t^2$. Let $ A \cup B$ be a partition of the vertex set of $K_{4n}$ such that $|A| = 2n + t - 1$ and $|B| = 2n - t + 1$.  We colour 1 all edges with one end vertex in $A$ and one in $B$, while the remaining edges are all coloured $-1$. Then there are $|A| |B| = (2n + t - 1)(2n - t + 1) = 4n^2 - t^2 + 2t -1$ edges coloured 1, and there are ${4n \choose 2} -  |A| |B|  = 2n(4n-1) - ( 4n^2 - n + 2 \sqrt{n} -1) = 4n^2 - n - 2t + 1$ edges coloured $-1$. That means that we have a difference of only $4t - 2 = 4 \sqrt{n} - 2$ between these two numbers. To see that there is no spanning zero-sum matching, observe that we need $n$ matching edges from each colour, meaning that we would have $n$ 1-coloured edges crossing between $A$ and $B$, which would leave $|A|-n = n+t-1= t^2 + t - 1 = t(t+1) - 1$ free vertices from $A$ and $|B| - n = n-t+1 =  t^2 - t + 1 = t(t-1) + 1$ vertices from $B$. However, since both numbers  $t(t+1) - 1$ and $t(t-1) + 1$ are odd, we can take at most $\frac{t(t+1)-2}{2} + \frac{t(t-1)}{2} = t^2 - 1 = n-1$ matching edges coloured -$1$. Hence, there is no spanning zero-sum matching.

So we pose the following problem:

\begin{problem}
Suppose $f : E(K_{4n}) \rightarrow \{ -1,1 \}$  is such that  $e(-1)= e(1)$ ---  does  a zero-sum matching always exist?

\end{problem}


\bibliographystyle{plain}
\bibliography{zerosumbib}
 
\end{document}